\documentclass{amsart}
\usepackage{setspace}
\usepackage{a4}
\usepackage{amssymb,amsmath,amsthm,latexsym}
\usepackage{amsfonts}
\usepackage{amsfonts}
\usepackage{graphicx}
\usepackage{textcomp}
\usepackage{cite}
\usepackage{enumerate}
\usepackage[mathscr]{euscript}
\usepackage{mathtools}
\newtheorem{theorem}{Theorem}[section]

\newtheorem{conjecture}[theorem]{Conjecture}

\newtheorem{remark}[theorem]{Remark}

\title{This is the title}
\usepackage{hyperref}
\usepackage{enumerate}
\usepackage{mathtools}
\usepackage{amsmath}
\usepackage{tikz}
\usepackage{hyperref}
\usepackage{graphicx}
\usepackage{textcomp}
\usetikzlibrary{cd}
\usepackage{fancyhdr}

\begin{document}
\begin{center}
{\bf{   C*-ALGEBRAIC SCHOENBERG CONJECTURE}\\
K. MAHESH KRISHNA}  \\
Post Doctoral Fellow \\
Statistics and Mathematics Unit\\
Indian Statistical Institute, Bangalore Centre\\
Karnataka 560 059 India\\
Email: kmaheshak@gmail.com\\

Date: \today
\end{center}

\hrule
\vspace{0.5cm}
\textbf{Abstract}: Based on Schoenberg conjecture \textit{[Amer. Math. Monthly., 1986]}/Malamud-Pereira theorem  \textit{[J. Math. Anal. Appl, 2003]}, \textit{[Trans. Amer. Math. Soc., 2005]} we formulate the following conjecture which we call C*-algebraic Schoenberg  Conjecture.\\
\textbf{ C*-algebraic Schoenberg Conjecture : Let  $\mathcal{A}$ be a   C*-algebra. Let $d\in \mathbb{N}\setminus\{1\}$, $P(z)	\coloneqq (z-a_1)(z-a_2)\cdots (z-a_d)$ be a polynomial over $\mathcal{A}$ with $a_1, a_2, \dots, a_d \in \mathcal{A} $. If  $P'$ can be written as   $P'(z)= d(z-b_1)(z-b_2)\cdots (z-b_{d-1})$ on $\mathcal{A}$ with $b_1, b_2, \dots, b_{d-1} \in \mathcal{A} $, then 
 	\begin{align*}
 		\sum_{k=1}^{d-1}b_kb_k^*\leq \frac{1}{d^2}\left[\sum_{j=1}^{d}a_j\right]\left[\sum_{j=1}^{d}a_j\right]^*+	\frac{d-2}{d}\sum_{j=1}^{d}a_ja_j^*
 	\end{align*}
 	and 
 	\begin{align*}
 		\sum_{k=1}^{d-1}b_k^*b_k\leq \frac{1}{d^2}\left[\sum_{j=1}^{d}a_j\right]^*\left[\sum_{j=1}^{d}a_j\right]+	\frac{d-2}{d}\sum_{j=1}^{d}a_j^*a_j.
 \end{align*}}
We show that  C*-algebraic Schoenberg conjecture  holds for  degree 2 C*-algebraic polynomials over   C*-algebras. 

\hrule
\tableofcontents
\hrule
\section{Introduction}
In 1986 Schoenberg made the following conjecture in the \textit{American Mathematical Monthly} which is known as Schoenberg conjecture \cite{SCHOENBERG1986}. 
\begin{conjecture} \cite{SCHOENBERG1986, MALAMUD, PEREIRA, DEBRUINIVANOVSHARMA}\label{SCHOENBERGCONJECTURE} \textbf{(Schoenberg Conjecture/Malamud-Pereira Theorem)
		Let $P(z)\coloneqq (z-a_1)(z-a_2)\cdots (z-a_d)$ be a degree $d\geq2$ polynomial over $ \mathbb{C}$. Let $b_1, b_2, \dots, b_{d-1}$ be zeros of   $P'$. 
		Then 
		\begin{align*}
			\sum_{k=1}^{d-1}|b_k|^2\leq \frac{1}{d^2}\left|\sum_{j=1}^{d}a_j\right|^2+	\frac{d-2}{d}\sum_{j=1}^{d}|a_j|^2.
	\end{align*}}
\end{conjecture}
The case $d=2$ follows from a computation. Schoenberg himself proved the case $d=3$ \cite{SCHOENBERG1986}. In 1996 Ivanov and Sharma proved Conjecture \ref{SCHOENBERGCONJECTURE}  for $d=4$ \cite{IVANOVSHARMA1996}. In 2003 Malamud (using inverse spectral theorem) \cite{MALAMUD, MALAMUD2003}  and independently  Pereira (using differentiators) \cite{PEREIRA} proved the Conjecture \ref{SCHOENBERGCONJECTURE} for all $d$. In 2006  Cheung and Ng  gave another proof of Schoenberg conjecture (using companion matrices)  \cite{CHEUNGNG}. In 2016 Kushel and Tyaglov gave one more  proof of Schoenberg conjecture (using circulant matrices) \cite{KUSHELTYAGLOV}.  In 1999 de Bruin and Sharma proposed a higher order version of Schoenberg conjecture \cite{DEBRUINSHARMA}. This was proved by Cheung and Ng in 2006 \cite{CHEUNGNG}. 
\begin{conjecture}\cite{DEBRUINSHARMA, CHEUNGNG}\label{DEBRUINSHARMACONJECTURE}
\textbf{(de Bruin-Sharma Conjecture/Cheung-Ng Theorem)
	Let $P(z)\coloneqq (z-a_1)(z-a_2)\cdots (z-a_d)$ be a degree $d\geq2$ polynomial over $ \mathbb{C}$. Let $b_1, b_2, \dots, b_{d-1}$ be zeros of   $P'$. 
	If 
		\begin{align*}
		\sum_{j=1}^{d}a_j=0, 	
	\end{align*}
	then 
	\begin{align*}
		\sum_{k=1}^{d-1}|b_k|^4\leq 	\frac{2}{d^2}\left(\sum_{j=1}^{d}|a_j|^2\right)^2+\frac{d-4}{d}\sum_{j=1}^{d}|a_j|^4.
\end{align*}}	
\end{conjecture}
In 2016 Kushel and Tyaglov improved Conjecture \ref{DEBRUINSHARMACONJECTURE} and derived the following theorem \cite{KUSHELTYAGLOV}.
\begin{theorem}\label{KUSHELTYAGLOVTHEOREM}\cite{KUSHELTYAGLOV} 
\textbf{(Kushel-Tyaglov Theorem)
	Let $P(z)\coloneqq (z-a_1)(z-a_2)\cdots (z-a_d)$ be a degree $d\geq2$ polynomial over $ \mathbb{C}$. Let $b_1, b_2, \dots, b_{d-1}$ be zeros of   $P'$. Then 	
		\begin{align*}
		\sum_{k=1}^{d-1}|b_k|^4&\leq 	\frac{d-6}{d}\sum_{j=1}^{d}|a_j|^4+\frac{1}{d^2}\left(\sum_{j=1}^{d}|a_j|^2\right)^2+\frac{1}{d^2}\left|\sum_{j=1}^{d}a_j^2-\frac{1}{d^2}\left(\sum_{k=1}^{d}a_k\right)^2\right|^2\\
		&\quad +\frac{2}{d}\sum_{j=1}^{d}|a_j|^2\left|a_j+\frac{1}{d}\sum_{k=1}^{d}a_k\right|^2-\frac{4}{d^3}\sum_{j=1}^{d}|a_j|^2\left|\sum_{k=1}^{d}a_k\right|^2.
	\end{align*}}
\end{theorem}
In this paper we formulate Conjecture \ref{SCHOENBERGCONJECTURE} for polynomials over  C*-algebra. We show in Theorem \ref{DEGREETWOHOLDS} that for degree 2 C*-algebraic polynomials our conjecture holds. We also formulate C*-algebraic versions of Conjecture \ref{DEBRUINSHARMACONJECTURE} and Theorem \ref{KUSHELTYAGLOVTHEOREM} (as a conjecture).

\section{C*-algebraic Schoenberg conjecture}
	Let $\mathcal{A}$ be a   C*-algebra. Given $P(z)	\coloneqq (z-a_1)(z-a_2)\cdots (z-a_d)$ for all  $z\in \mathcal{A}$ with  $a_1, a_2, \dots, a_d \in \mathcal{A} $, we define 
\begin{align*}
	P'(z)=\sum_{j=1}^{d}(z-a_1)\cdots \widehat{(z-a_j)}\cdots (z-a_d), \quad \forall z \in \mathcal{A}
\end{align*}
where the term with cap is missing.
We state C*-algebraic version of Conjecture \ref{SCHOENBERGCONJECTURE} as follows.
\begin{conjecture}\label{CALGEBRAICSCHOENBERG} 
	\textbf{(C*-algebraic Schoenberg Conjecture)
		Let  $\mathcal{A}$ be a   C*-algebra. Let $d\in \mathbb{N}\setminus\{1\}$, $P(z)	\coloneqq (z-a_1)(z-a_2)\cdots (z-a_d)$ be a polynomial over $\mathcal{A}$ with $a_1, a_2, \dots, a_d \in \mathcal{A} $. If  $P'$ can be written as   $P'(z)= d(z-b_1)(z-b_2)\cdots (z-b_{d-1})$ on $\mathcal{A}$ with $b_1, b_2, \dots, b_{d-1} \in \mathcal{A} $, 
		then 
		\begin{align*}
			\sum_{k=1}^{d-1}b_kb_k^*\leq \frac{1}{d^2}\left[\sum_{j=1}^{d}a_j\right]\left[\sum_{j=1}^{d}a_j\right]^*+	\frac{d-2}{d}\sum_{j=1}^{d}a_ja_j^*
		\end{align*}
		and 
		\begin{align*}
			\sum_{k=1}^{d-1}b_k^*b_k\leq \frac{1}{d^2}\left[\sum_{j=1}^{d}a_j\right]^*\left[\sum_{j=1}^{d}a_j\right]+	\frac{d-2}{d}\sum_{j=1}^{d}a_j^*a_j.
	\end{align*}}	
\end{conjecture}
\begin{theorem}\label{DEGREETWOHOLDS}
	Conjecture \ref{CALGEBRAICSCHOENBERG}	holds for C*-algebraic polynomials of degree 2. 
\end{theorem}
\begin{proof}
	Let $\mathcal{A}$ be a   C*-algebra and  $P(z)\coloneqq (z-a)(z-b)$ be a degree 2 polynomial over $\mathcal{A}$, $a, b \in \mathcal{A} $. Then the zero of $P'$ is $c=\frac{a+b}{2} $. Therefore 
	\begin{align*}
		cc^*=\frac{(a+b)(a^*+b^*)}{4}=\frac{(a+b)(a^*+b^*)}{4}+\frac{2-2}{2} (aa^*+bb^*)
	\end{align*}	
	and 
	\begin{align*}
		c^*c=\frac{(a^*+b^*)(a+b)}{4}=\frac{(a^*+b^*)(a+b)}{4}+\frac{2-2}{2} (a^*a+b^*b).
	\end{align*}
Therefore both inequalities in Conjecture \ref{CALGEBRAICSCHOENBERG} hold with equality.
\end{proof}
  We state the C*-algebraic version of Conjecture \ref{DEBRUINSHARMACONJECTURE} as follows. 
\begin{conjecture}\label{BSCONJECTURE}
	\textbf{(C*-algebraic de Bruin-Sharma Conjecture)
		Let $\mathcal{A}$ be a  C*-algebra, $d\in \mathbb{N}\setminus\{1\}$ and let $P(z)	\coloneqq (z-a_1)(z-a_2)\cdots (z-a_d)$ be a polynomial over $\mathcal{A}$ with $a_1, a_2, \dots, a_d \in \mathcal{A}$. Assume that $P'$ can be written as   $P'(z)\coloneqq d (z-b_1)\cdots (z-b_{d-1})$ on $\mathcal{A}$ with $b_1, b_2, \dots, b_{d-1} \in \mathcal{A}$. If
		\begin{align*}
			\sum_{j=1}^{d}a_j=0, 	
		\end{align*}
		then 
		\begin{align*}
			\sum_{k=1}^{d-1}(b_kb_k^*)^2\leq \frac{2}{d^2}\left(\sum_{j=1}^{d}a_ja_j^*\right)^2+	\frac{d-4}{d}\sum_{j=1}^{d}(a_ja_j^*)^2
		\end{align*}
		and 
		\begin{align*}
			\sum_{k=1}^{d-1}(b_k^*b_k)^2\leq \frac{2}{d^2}\left(\sum_{j=1}^{d}a_j^*a_j\right)^2+	\frac{d-4}{d}\sum_{j=1}^{d}(a_j^*a_j)^2.
	\end{align*}}
\end{conjecture}
It is clear that Conjecture \ref{BSCONJECTURE} holds for C*-algebraic polynomials of degree 2. We next formulate a strengthening of  Conjecture \ref{BSCONJECTURE} based on  Theorem \ref{KUSHELTYAGLOVTHEOREM}.
\begin{conjecture}\label{KTC}
	\textbf{(C*-algebraic Kushel-Tyaglov Conjecture)
		Let $\mathcal{A}$ be a  C*-algebra, $n\in \mathbb{N}\setminus\{1\}$ and let $P(z)	\coloneqq (z-a_1)(z-a_2)\cdots (z-a_d)$ be a polynomial over $\mathcal{A}$ with $a_1, a_2, \dots, a_d\in \mathcal{A}$. Assume that $P'$ can be written as   $P'(z)\coloneqq d (z-b_1)\cdots (z-b_{d-1})$ on $\mathcal{A}$ with $b_1, b_2, \dots, b_{d-1} \in \mathcal{A}$. Then 
		\begin{align*}
			\sum_{k=1}^{d-1}(b_kb_k^*)^2&\leq 	\frac{d-6}{d}\sum_{j=1}^{n}(a_ja_j^*)^2+\frac{1}{d^2}\left(\sum_{j=1}^{d}a_ja_j^*\right)^2+
			\frac{1}{d^2}\left[\sum_{j=1}^{d}a_j^2-\frac{1}{d^2}\left(\sum_{k=1}^{d}a_k\right)^2\right] \left[\sum_{j=1}^{d}a_j^2-\frac{1}{d^2}\left(\sum_{k=1}^{d}a_k\right)^2\right]^*\\
			&+\frac{2}{d}\sum_{j=1}^{d}a_j\left[a_j+\frac{1}{d}\sum_{k=1}^{d}a_k\right]\left[a_j+\frac{1}{d}\sum_{k=1}^{d}a_k\right]^*a_j^*-\frac{4}{d^3}\sum_{j=1}^{d}a_j \left[\sum_{k=1}^{d}a_k \right] \left[\sum_{k=1}^{d}a_k \right]^*a_j^*
		\end{align*}
		and 
		\begin{align*}
			\sum_{k=1}^{d-1}(b_k^*b_k)^2&\leq 	\frac{d-6}{d}\sum_{j=1}^{d}(a_j^*a_j)^2+\frac{1}{d^2}\left(\sum_{j=1}^{d}a_j^*a_j\right)^2+
			\frac{1}{d^2}\left[\sum_{j=1}^{d}a_j^2-\frac{1}{d^2}\left(\sum_{k=1}^{d}a_k\right)^2\right]^* \left[\sum_{j=1}^{d}a_j^2-\frac{1}{d^2}\left(\sum_{k=1}^{d}a_k\right)^2\right]\\
			&+\frac{2}{d}\sum_{j=1}^{n}a_j^*\left[a_j+\frac{1}{d}\sum_{k=1}^{d}a_k\right]^*\left[a_j+\frac{1}{d}\sum_{k=1}^{d}a_k\right]a_j-\frac{4}{d^3}\sum_{j=1}^{d}a_j^* \left[\sum_{k=1}^{d}a_k \right]^* \left[\sum_{k=1}^{d}a_k \right]a_j.\\
	\end{align*}}		
\end{conjecture}
\begin{remark}
\textbf{C*-algebraic Sendov conjecture} has been formulated in 	\cite{MAHESHKRISHNA}.
\end{remark}

 \bibliographystyle{plain}
 \bibliography{reference.bib}

\end{document}